\documentclass[12pt, reqno]{amsart}
\usepackage{graphicx}
\usepackage{amsmath,amssymb,mathrsfs}
\newtheorem{thm}{Theorem}[section]

\newtheorem{lem}[thm]{Lemma}

\theoremstyle{definition}
\newtheorem{defn}[thm]{Definition}
\theoremstyle{remark}
\newtheorem{rem}[thm]{Remark}
\numberwithin{equation}{section}

\begin{document}

\title[Approximate Invisibility Cloaking]{Approximate Acoustic Cloaking in Inhomogeneous Isotropic Space}%
\author{Liu Hongyu}%
\address{Department of Mathematics and Statistics, University of North Carolina, Charlotte, NC 28223}%
\email{hongyu.liuip@gmail.com}%

\thanks{Research was partly supported by NSF grant DMS 0724808.}%

\begin{abstract}
In this paper, we consider the approximate acoustic cloaking in
inhomogeneous isotropic background space. By employing
transformation media, together with the use of a sound-soft layer
lining right outside the cloaked region, we show that one can
achieve the near-invisibility by the `blow-up-a-small-region'
construction. This is based on novel scattering estimates
corresponding to small sound-soft obstacles located in isotropic
space. One of the major novelties of our scattering estimates is that one cannot
make use of the scaling argument in the setting of current study due to the simultaneous presence of asymptotically small obstacle components and regularly sized obstacle components, and one has to decouple
the nonlinear scattering interaction between the small obstacle components and, the regular obstacle components
together with the background medium.


%
\end{abstract}
\maketitle
\section{Introduction}

A region is said to be \emph{cloaked} if its contents together with
the cloak are indistinguishable from the background space to certain
exterior detections. Blueprints for making objects invisible to
electromagnetic waves were proposed by Pendry {\it et al.}
\cite{PenSchSmi} and Leonhardt \cite{Leo} in 2006. In the case of
electrostatics, the same idea was discussed by Greenleaf {\it et
al.} \cite{GLU2} in 2003. The key ingredient for those constructions
is that optical parameters have transformation properties and could
be {\it push-forwarded} to form new material parameters. The
obtained materials/media are called {\it transformation media},
which we shall further examine in the current work for approximate
acoustic cloaking in inhomogeneous isotropic space.

The transformation media proposed in \cite{GLU2,PenSchSmi} are
rather singular. This poses much challenge to both theoretical
analysis and practical fabrication. In order to avoid the singular
structures, several regularized approximate cloaking schemes are
proposed in \cite{GKLU_2,KOVW,KSVW,Liu,RYNQ}. The basic idea is to
introduce regularization into the singular transformation underlying
the ideal cloaking, and instead of the perfect invisibility, one
would consider the `near-invisibility' depending on the regularizer.
The works \cite{GKLU_2} and \cite{RYNQ} are based on truncation,
whereas in \cite{KOVW,KSVW,Liu}, the `blow-up-a-point'
transformation in \cite{GLU2,PenSchSmi} is regularized to be the
`blow-up-a-small-region' transformation. The performances of both
regularization schemes have been assessed for cloaking of acoustic
waves to give successful near-invisibility effects. Particularly, in
\cite{KOVW}, the authors show that in order to `nearly-cloak' an
{\it arbitrary} content, it is necessary to employ an absorbing
(`lossy') layer lining right outside the cloaked region. Since
otherwise, there exist cloaking-busting inclusions which defy any
attempts of cloaking. If one lets the lossy parameter go to
infinity, heuristically we would have the sound-soft material
lining, which is the one considered in \cite{Liu}. All the
aforementioned studies for approximate acoustic cloaking are
conducted in the homogeneous background space, and there is no
result available in literature for the more general case when the
background space is allowed to be inhomogeneous. Nonetheless, it is worthy noting
that the result in \cite{KSVW} for approximate cloaking of conductivity equation could be
readily extended to the case with inhomogeneous background space by making use of the estimates in
\cite{FreVog} for small extreme conductivities; whereas the result in \cite{KOVW} could also
be readily extended to the case with inhomogeneous background space by making use of the low-frequency estimates
in \cite{Amm} for the reduced wave equation. One of the key points for those extensions is that the
corresponding studies could be reduced to the scattering estimates due to {\it uniformly} small
objects, and then by scaling arguments, the studies could be further reduced to those having been
considered in \cite{FreVog} and \cite{Amm}.

In this work, we shall consider the approximate cloaking for acoustic
waves in a very general and practical setting when the background space is allowed to be inhomogeneous but
isotropic. We are mainly interested in the practical case that in the inhomogeneous space, there are both
target objects one intends to cloak and non-target objects being uncloaked.
By transformation-optics-approach, we construct the
approximate cloaking devices by the `blow-up-a-small-region' scheme.
In order to overcome the cloaking-busts due to resonance, we
implement the sound-soft lining right outside of the cloaked region. In assessing the
cloaking performance, the study is shown to be reduced to the scattering estimate due
to extended objects with both asymptotically small obstacle components and regularly sized
obstacle components being presented simultaneously in an inhomogeneous space. One cannot make use of the scaling arguments
as mentioned earlier and has to decouple the nonlinear scattering interactions between the small obstacle components,
the regular obstacle components and the background medium. By nonlinear scattering interaction we mean that the interaction between the scattered wave fields due to different components in a scattering system is a nonlinear process. Similar case with the background space
being homogeneous has been investigated in \cite{Liu}, where boundary integral equations method is used to decouple
the scattering due to the small obstacle components and regular obstacle components. For the current study, the presence of the
inhomogeneous medium makes the corresponding arguments rather technical. We derive a novel system of
integral equations underlying the scattering problem, which combines
the volume potential operator of Lippman-Schwinger type and single-
and double- boundary layer potential operators. An asymptotical
coupling parameter is also a crucial incorporation. Since the scaling arguments does not apply here,
another major difficulty one need to handle is that the domain of the underlying
PDE is always in change in the asymptotic analysis. By extensive use
of the potential operators theory,
we show that the scattering contribution from small obstacle
components is also asymptotically small with respect to their sizes,
which justifies the near-cloak.

In this paper, we focus entirely on transformation-optics-approach
in constructing cloaking devices. We refer to \cite{GKLU4,GKLU5,Nor,U2,YYQ}
for state-of-the-art surveys on the rapidly growing literature and
many striking applications of transformation optics. But we would
also like to mention in passing the other promising cloaking schemes
including the one based on anomalous localized resonance \cite{MN},
and another one based on special (object-dependent) coatings
\cite{AE}.

The rest of the paper is organized as follows. In Section 2, we give
a brief discussion on inverse acoustic scattering and invisibility
cloaking. In Section 3, we collect the basics on transformation
optics and apply them to the construction of approximate cloaking
devices. Sections 4 is devoted to the scattering estimates due to extended objects
and the proof of the main result on approximate cloaking.

\section{Acoustic scattering and invisibility cloaking}

We first fix some notations that shall be used throughout the rest
of the paper. For two domains $D$ and $\Omega$ in $\mathbb{R}^n$
(n=2,3), $D\Subset\Omega$ means that $D\subset\bar{D}\subset\Omega$.
$\chi_D$ denotes the characteristic function of $D$. Also, $B_r$ is
reserved for a central ball of radius $r$ in $\mathbb{R}^n$. For two
relations $\mathcal{R}_1$ and $\mathcal{R}_2$,
$\mathcal{R}_1\lesssim\mathcal{R}_2$ will refer to
$\mathcal{R}_1\leq c\mathcal{R}_2$ with $c$ a generic constant which
may change in different inequalities but must be fixed and
finite in a given relation. $\mathcal{R}_1\sim\mathcal{R}_2$
means that we have both $\mathcal{R}_1\lesssim\mathcal{R}_2$ and
$\mathcal{R}_2\lesssim\mathcal{R}_1$. We shall also let
$c(i_1,i_2,\ldots,i_l)$ denote a generic constant that depends on
the ingredients $i_1,i_2,\ldots,i_l$. The notations of function
spaces is the usual one, e.g., $L^2(\Omega)$ denotes the space of
square integrable functions on $\Omega$, and $H^s(\Omega)$ denotes
the Sobolev space of order $s$ on $\Omega$.

Let $\sigma=(\sigma^{ij})_{i,j=1}^n\in\mbox{Sym}_n$ be a
symmetric-matrix-valued function on $\mathbb{R}^n$ and bounded in
the sense that, for some constants $0<c_0, C_0<\infty$,
\begin{equation}
\label{eqn:Bound_Sigma} c_0 \xi^T \xi \leq \xi^T \sigma(x) \xi \leq
C_0 \xi^T \xi
\end{equation}
for all $x\in \mathbb{R}^n$ and $\xi \in \mathbb{R}^n$. Also, we let
$\eta\in L^\infty(\mathbb{R}^n)$ be a scalar function with
\begin{equation}
\label{eqn:Lower_Bound_Lambda} \eta(x) \geq \eta_0 > 0 \quad \forall
\ x \in \mathbb{R}^n .
\end{equation}
In acoustic scattering, $\sigma^{-1}$ and $\eta$, respectively,
represent the mass density tensor and the bulk modulus of a {\it
regular} acoustic medium. Starting from now on, we denote by
$\{\mathbb{R}^n; \sigma, \eta\}$ an acoustic medium as described
above. We shall assume that the inhomogeneity of the acoustic medium
is compactly supported, namely, $\sigma=I$ and $\eta=1$ in
$\mathbb{R}^n\backslash\bar{\mathbf{B}}$ with $\mathbf{B}$ a bounded
domain in $\mathbb{R}^n$. In $\mathbb{R}^n$, the time-harmonic
acoustic wave propagation is governed by the heterogeneous Helmholtz
equation
\begin{equation}\label{eq:Helmholtz II}
\sum_{i,j=1}^{n}\partial_i(\sigma^{ij}\partial_j u)+{\omega^2}\eta
u= f,
\end{equation}
where $supp\ f\Subset\mathbf{B}$ represents a compactly supported
source/sink. Here, $\omega>0$ represents the frequency of the wave
scattering. In the following, we shall write
$\{\mathbf{B};\sigma,\eta,f\}$ to denote the scattering object
including the inhomogeneous medium and the source/sink supported in
$\Omega$.
We seek solutions admitting the following asymptotic development as
$|x|\rightarrow+\infty$
\begin{equation}\label{eq:asymptotic}
u(x)=e^{ix\cdot\xi}+\frac{e^{i\omega|x|}}{|x|^{(n-1)/2}}\left\{A(\theta,\theta',\omega)+\mathcal{O}(\frac{1}{|x|})\right\},
\end{equation}
where $\theta,\theta'\in\mathbb{S}^{n-1}$ and $\xi=\omega\theta'$.
$A(\theta,\theta',\omega)$ is known as the {\it scattering
amplitude}, which depends on the direction $\theta'$ and frequency
$\omega$ of the incident wave $u^i:=e^{ix\cdot\xi}$, observation
direction $\theta$, and obviously, also the underlying scattering
object $\{\Omega;\sigma,\eta,f\}$. In inverse scattering theory, one
intends to recover the target object $\{\mathbf{B};\sigma,\eta,f\}$
by knowledge of $A(\theta,\theta',\omega)$ (cf. \cite{ColKre,
LaxPhi}). In the sequel, we use $A(\{\mathbf{B}; \sigma,\eta, f\})$
to denote the scattering amplitude corresponding to
$\{\mathbf{B};\sigma,\eta,f\}$. In this context, an invisibility
cloaking device is introduced as follows.

\begin{defn}\label{def:cloaking device}
For a given regular background/reference medium $\{\mathbf{B};$
$\sigma_0,\eta_0\}$, let $\Omega$ and $D$ be bounded domains such
that $D\Subset\Omega\Subset\mathbf{B}$. $\Omega\backslash\bar{D}$
and $D$ represent, respectively, the cloaking region and the cloaked
region. $\{\Omega\backslash\bar{D}; \sigma_c,\eta_c\}$ is said to be
an {\it invisibility cloaking device} for the region $D$ with
respect to the background space $\{\mathbf{B}; \sigma_0,\eta_0\}$ if
\[
A(\{\mathbf{B};\sigma_e,\eta_e,f_e\})=A(\{\mathbf{B};\sigma_0,\eta_0\}),
\]
where the extended object
$\{\mathbf{B};\sigma_e,\eta_e,f_e\}=\{D;\sigma_a,\eta_a,f_a\}\oplus\{\Omega\backslash\bar{D};
\sigma_c,\eta_c\}$ $\oplus\{\mathbf{B}\backslash
\bar{\Omega};\sigma_0,\eta_0\}$ with $\{D; \sigma_a,\eta_a,f_a\}$
arbitrary but regular. Here, we use $\oplus$ to concatenate separate
scattering objects.
\end{defn}

According to Definition~\ref{def:cloaking device}, the cloaking
medium $\{\Omega\backslash\bar{D}; \sigma_c,\eta_c\}$ makes the
target object $\{D;\sigma_a,\eta_a,f_a\}$ indistinguishable from the
background space $\{\mathbf{B};\sigma_0,\eta_0\}$, and thus
invisible to the scattering detections.

As can be seen from our subsequent discussion in
Section~\ref{sect:3}, one has to implement singular cloaking medium
in order to achieve the ideal cloaking (see also the related
discussion in \cite{KOVW} and \cite{GKLU_2}). This poses many
challenges to both mathematical analysis and physical realization.
In order to construct practical nonsingular cloaking devices, it is
natural to incorporate regularization by considering approximate
cloaking. We conclude this section by introducing the notion of
approximate acoustic cloaking.

\begin{defn}\label{def:2}
Let $\{\mathbf{B}; \sigma_0,\eta_0\}$, $\Omega$ and $D$ be given as
in Definition~\ref{def:cloaking device}. Let $\rho>0$ denote a
regularizer and $e(\rho)$ be a positive function such that
\[
e(\rho)\rightarrow 0\quad\mbox{as\ $\rho\rightarrow 0^+$}.
\]
$\{\Omega\backslash\bar{D}; \sigma_c^\rho,\eta_c^\rho\}$ is said to
be an {\it approximate invisibility cloaking} for the region $D$ if
\begin{equation}\label{eq:near invisibility}
\|A(\{\mathbf{B};\sigma_e,\eta_e,f_e\})-A(\{\mathbf{B};\sigma_0,\eta_0\})\|
=e(\rho)\quad\mbox{as $\rho\rightarrow 0^+$},
\end{equation}
where the extended object $\{\mathbf{B};\sigma_e,\eta_e,f_e\}$ is
defined similarly to the one in Definition~\ref{def:cloaking device}
by replacing $\{\Omega\backslash\bar{D};\sigma_c,\eta_c\}$ with
$\{\Omega\backslash\bar{D};\sigma_c^\rho,\eta_c^\rho\}$.
\end{defn}

According to Definition~\ref{def:2}, with the cloaking device
$\{\Omega\backslash\bar{D}; \sigma_c^\rho,\eta_c^\rho\}$ we shall
have the `near-invisibility' cloaking effect within $e(\rho)$
depending on the regularizer $\rho$.

\section{Transformation optics and cloaking by sound-soft layer
lining}\label{sect:3}

The transformation-optics-approach for the construction of cloaking
devices critically relies on the following transformation invariance
of the Helmh- oltz equation. Let $\tilde
x=F(x):\Omega\rightarrow\tilde\Omega$ be a bi-Lipschitz and
orientation-preserving mapping. For an acoustic medium
$\{\Omega;\sigma,\eta\}$, we let the {\it push-forwarded} medium be
\begin{equation}\label{eq:pushforward}
\{\tilde\Omega;\tilde\sigma,\tilde\eta\}=F_*\{\Omega;\sigma,\eta\}:=\{\Omega;
F_*\sigma, F_*\eta\},
\end{equation}
where
\begin{equation}
\begin{split}
&\tilde{\sigma}(\tilde
x)=F_*\sigma(x):=\frac{1}{J}M\sigma(x)M^T|_{x=F^{-1}(\tilde x)}\\
&\tilde{\eta}(\tilde x)=F_*\eta(x):=\eta(x)/J|_{x=F^{-1}(\tilde x)}
\end{split}
\end{equation}
and $M=(\partial \tilde{x}_i/\partial x_j)_{i,j=1}^n$,
$J=\mbox{det}(M)$. Assume that $u\in H^1(\Omega)$ is a solution to
the Helmholtz equation associated with $\{\Omega;\sigma,\eta\}$,
namely,
\[
\nabla\cdot(\sigma\nabla u)+\omega^2\eta u=0\quad\mbox{on\
$\Omega$},
\]
then the pull-back field $\tilde u=(F^{-1})^*u:=u\circ F^{-1}\in
H^1(\tilde\Omega)$ verifies
\[
\tilde{\nabla}\cdot(\tilde\sigma\tilde\nabla \tilde
u)+\omega^2\tilde\eta\tilde u=0\quad\mbox{on\ $\tilde\Omega$},
\]
where we use $\nabla$ and $\tilde\nabla$ to distinguish the
differentiations respectively in $x$- and $\tilde x$-coordinates. We
refer to \cite{KOVW, Liu} for a proof of this invariance.

Next, we shall give the construction of the cloaking device. We
first fix the setting for our study. Let
$\{\mathbb{R}^n;\sigma_0,\eta_0\}$ be a regular background/reference
space. Throughout, we shall assume that the background space medium
is isotropic, i.e., $\sigma_0$ is a multiple of a scalar function
and the identity matrix. In the following, we treat $\sigma_0$ as a
scalar function which should be clear in the context. Let
$\sigma_0(x)\in C^3(\mathbb{R}^n)$ and $\eta_0\in
C^1(\mathbb{R}^n)$, and $supp\ (1-\sigma_0), supp\
(1-\eta_0)\Subset\mathbf{B}$. Let $D, G\Subset\mathbf{B}$ be
$C^2$-domains in $\mathbb{R}^n$ such that
$\mathbb{R}^n\backslash(\bar{D}\cup \bar{G})$ is connected and $D$
is of the form $D=\cup_{k=1}^l D_k$, where each $D_k$ is simply
connected. It is assumed that $D_k$'s are disjoint and separate from
$G$. Let $z_k\in D_k$ be points such that
\begin{equation}
|z_k-z_{k'}|\geq d_0>0,\quad \forall\ k\neq k',
\end{equation}
and let $P_k$ and $O_k$ be $C^2$-domains such that $P_k$'s and
$O_k$'s are simply connected and
\[
z_k\in P_k\Subset D_k,\ \ \mathbf{0}\in O_k.
\]
Define
\begin{equation}
U=\bigcup_{k=1}^l U_k,\quad U_k=z_k+O_k.
\end{equation}
Here, we assume that $\mathbb{R}^n\backslash(\bar{U}\cup\bar{G})$ is
connected and
\begin{equation}
\mbox{dist}_{\mathbb{R}^n}(\bar{U},\bar{G})\geq r_0>0,\quad
\mbox{dist}_{\mathbb{R}^n}(\bar{U}_{k},\bar{U}_{k'})\geq
\tilde{r}_0>0,\ \ \forall\ k\neq k'.
\end{equation}
For a sufficiently small $\rho>0$, we let
\begin{equation}
U_\rho=\bigcup_{k=1}^l U_{k,\rho},\quad U_{k,\rho}=z_k+\rho O_k.
\end{equation}
Clearly, the parameter $\rho$ determines the relative size of
$U_{k,\rho}$. Let $\rho$ be sufficiently small such that
$U_{k,\rho}\Subset P_k$. We suppose that there are
orientation-preserving $C^2$-diffeomorphisms $F_k^\rho:
\bar{D}_k\backslash U_{k,\rho}\rightarrow \bar{D}_k\backslash P_k$
such that
\begin{equation}
\bar{D}_k\backslash P_k=F_k^\rho(\bar{D}_k\backslash U_{k,\rho}),\ \
F_k^\rho|_{\partial D_k}=\mbox{Identity},\quad k=1,2,\ldots,l.
\end{equation}
One sees that $F_k^\rho$ blows up the small region $U_{k,\rho}$ to
$P_k$ within $\bar{D}_k$. Here, we would like to mention that one of
such mappings is given by
\[
\mathcal{F}(x)=(a+b|x|)\frac{x}{|x|},\ \
a=\frac{r_1-\rho}{r_2-\rho}r_2,\ \ b=\frac{r_2-r_1}{r_2-\rho},\ \
r_2>r_1>0,
\]
which blows up the central ball of radius $\rho$ to the central ball
of radius $r_1$ within the central ball of radius $r_2$.

We are ready to construct the approximate cloaking device. Let
$P_k$, $D_k\backslash\bar{P}_k$ denote respectively the cloaked and
cloaking regions. The cloaking media are given by
\begin{equation}
\{D_k\backslash\bar{P}_k;\sigma_{k,c},\eta_{k,c}\}=(F_k^\rho)^*\{D_k\backslash\bar{U}_{k,\rho};\sigma_0,\eta_0\},\
\ 1\leq k\leq l.
\end{equation}
That is, we blow up a small region $U_{k,\rho}$ within $D_k$, and
derive the cloaking medium in $D_k\backslash\bar{P}_k$ by
pushing-forward the background medium in
$D_k\backslash\bar{U}_{k,\rho}$. Next, one can put the target
objects in $P_k$, namely, the passive {\it regular} media or
radiating sources. The last ingredient for our cloaking construction
is to introduce a thin sound-soft layer lining right outside $P_k$.
Then, $P_k$ can be regarded as a {\it sound-soft obstacle}. Here, by
a sound-soft obstacle, we mean a scattering object which prevents
acoustic wave penetrating inside (similarly, the wave propagation
inside $P_k$ cannot penetrate outside and completely trapped
inside), and the wave pressure vanishes on its boundary; that is,
one would the following homogeneous Dirichlet boundary condition
\[
u=0\quad \mbox{on\ \ $\partial P_k$}.
\]
For such construction, the target object together with its cloaking is
\begin{equation}
\mathcal{S}:=\bigoplus_{k=1}^l\{D_k\backslash\bar{P}_k;\sigma_{k,c},\eta_{k,c}\}\oplus
P,
\end{equation}
where $P=\cup_{k=1}^l P_k$ is the union of all target sound-soft obstacle
components; and the extended scattering object is given by
\begin{equation}
\mathcal{E}=\mathcal{S}\oplus{G}\oplus\{\mathbf{B}\backslash(\bar{D}\cup\bar{G});\sigma_0,\eta_0\}
\end{equation}

The near-invisibility of our construction is given in
the following theorem.

\begin{thm}\label{thm:approximate cloaking}
For any fixed $\xi\in\mathbb{R}^n$, we have as $\rho\rightarrow
0^+$,
\begin{equation}
\|A(\theta,\xi;\mathcal{E})-A(\theta,\xi;
G\oplus\{\mathbf{B}\backslash\bar{G};\sigma_0,\eta_0\})\|_{L^2(\mathbb{S}^{n-1})}=\mathcal{O}(e(\rho)),
\end{equation}
where
\begin{equation}\label{eq:e(rho)}
e(\rho)=\begin{cases} \rho,\ \ \ & n=3\\
(\ln \rho)^{-1},\ & n=2
\end{cases}
\end{equation}
\end{thm}

Theorem~\ref{thm:approximate cloaking} indicates that the cloaked components $P$ shall be nearly cloaked,
whereas the uncloaked component $G$ remains unaffected, even though
there is scattering interactions among different components.

We shall prove Theorem~\ref{thm:approximate cloaking} in the next
section, which is a consequence of the following theorem.

\begin{thm}\label{thm:small}
Let $U_\rho$ and $G$ be sound-soft obstacles and
\begin{equation}\label{eq:T}
\mathcal{T}:=U_\rho\oplus
G\oplus\{\mathbf{B}\backslash(\bar{U}_\rho\cup\bar{G});\sigma_0,\eta_0\}.
\end{equation}
For any fixed $\xi\in\mathbb{R}^n$, we have as $\rho\rightarrow 0^+$
\begin{equation}\label{eq:farfield difference0}
\|A(\theta,\xi;\mathcal{T})-A(\theta,\xi;G\oplus\{\mathbf{B}\backslash\bar{G};\sigma_0,\eta_0\})
\|_{L^2(\mathbb{S}^{n-1})}=\mathcal{O}(e(\rho)),
\end{equation}
where $e(\rho)$ is given in (\ref{eq:e(rho)}).
\end{thm}

Theorem~\ref{thm:small} indicates that for scattering due to
obstacles in the space $\{\mathbf{B};\sigma_0,\eta_0\}$, the
scattering contribution from small obstacle components is also small
in terms of their sizes. The proof of Theorem~\ref{thm:small} will
also be given in the next section. To our best knowledge, the result in Theorem~\ref{thm:small} are completely new
in literature. We would like to remark that the
scattering estimates due to small inclusions have received extensive
mathematical interests in literature (see, e.g. \cite{HamKan}). So,
our result in Theorem~\ref{thm:small} together with the techniques
developed for its proof is of mathematical importance and
significance for its own sake.

\section{Scattering estimates of small obstacles in isotropic space}

We first prove Theorem~\ref{thm:small}. The scattering problem
corresponding to $\mathcal{T}$ in (\ref{eq:T}) is given by
\begin{equation}\label{eq:scattering1}
\begin{cases}
& \mbox{div}(\sigma_0\nabla u)+\omega^2\eta_0 u=0\quad\mbox{in \ $\mathbb{R}^n\backslash(\bar{U}_\rho\cup\bar{G})$},\\
& u|_{\partial U_\rho\cup\partial G}=0,
\end{cases}
\end{equation}
whereas the one without the small inclusions $U_\rho$ is
\begin{equation}\label{eq:scattering2}
\begin{cases}
& \mbox{div}(\sigma_0\nabla \tilde u)+\omega^2\eta_0 \tilde u=0\quad\mbox{in \ $\mathbb{R}^n\backslash \bar{G}$},\\
& \tilde u|_{\partial G}=0,
\end{cases}
\end{equation}
Let $u^s(x):=u(x)-e^{ix\cdot\xi}$ and $\tilde{u}^s(x):=\tilde
u(x)-e^{ix\cdot\xi}$ be the so-called {\it scattered wave fields},
and $\mathbf{w}^s$ denote either $u^s$ or $\tilde u^s$. Then, it is
well known that $\mathbf{w}^s$ satisfies the Sommerfeld radiation
condition (cf. \cite{ColKre})
\begin{equation}\label{eq:radiation}
\lim_{|x|\rightarrow\infty}|x|^{(n-1)/2}\{\frac{\partial
\mathbf{w}^s}{\partial |x|}-i\omega \mathbf{w}^s\}=0.
\end{equation}
We refer to \cite{Isa2,Isa} for related study on the unique
existence of $u\in H_{loc}^1(\mathbb{R}^n\backslash(\bar{U}_\rho$
$\cup\bar{G}))$ to (\ref{eq:scattering1})-(\ref{eq:radiation}), and
$\tilde{u}\in H_{loc}^1(\mathbb{R}^n\backslash\bar{G})$ to
(\ref{eq:scattering2})-(\ref{eq:radiation}). However, for our
subsequent discussion, we shall derive integral representations to
$u$ and $\tilde u$. To that end, we first recall the following
celebrated gauge transformation (see, e.g. \cite{U}). Let $L_\sigma
u=\mbox{div}(\sigma\nabla u)$. One has
\begin{equation}\label{eq:gauge trans}
\sigma^{-1/2}\circ L_\sigma\circ\sigma^{-1/2}=\Delta-\gamma,\quad
\gamma=\sigma^{-1/2}\Delta\sigma^{1/2}.
\end{equation}
Setting $v=\sigma_0^{1/2}u$ and using
(\ref{eq:gauge trans}), (\ref{eq:scattering1}) and
(\ref{eq:radiation}) are transformed into
\begin{equation}\label{eq:transformed scattering}
\begin{cases}
& (\Delta+\omega^2q)v=0\quad\mbox{in\ \
$\mathbb{R}^n\backslash(\bar{U}_\rho\cup\bar{G})$},\\
& v|_{\partial U_\rho\cup\partial
G}=0,\\
\end{cases}
\end{equation}
where
\begin{equation}
q=\sigma_0^{-1}\eta_0-\omega^{-2}\gamma_0,\quad
\gamma_0=\sigma_0^{-1/2}\Delta\sigma_0^{1/2}.
\end{equation}
Moreover, $v^s:=v-e^{ix\cdot \xi}$ satisfies the radiating condition \eqref{eq:radiation}.
Obviously, $q\in C^1(\mathbb{R}^n)$ and $1-q$ is compactly supported
in $\mathbf{B}$. By introducing $\tilde v=\sigma_0^{1/2}\tilde u$
and $\tilde{v}^s=\tilde{v}-e^{ix\cdot\xi}$, we would have a
similar system for $\tilde{v}$ and $\tilde{v}_s$, which should be
clear in the context.

We shall make essential use of integral representations to $v$ and
$\tilde{v}$ for our arguments. To that end, we briefly introduce
some potential operators and fix some notations for our subsequent
discussion. Let $\Omega$ be a bounded $C^2$-domain in
$\mathbb{R}^n$. We write $\Omega^-$ as $\Omega$ and $\Omega^+$ as
its complementary, and $\Pi=\partial\Omega^+=\partial\Omega^-$. The
one-sided trace operators for $\Omega^+$ and $\Omega^-$ are denoted
by $\gamma^+$ and $\gamma^-$, respectively. The normal derivative of
a function $u\in H^1(\Omega^\pm)$ is understood in the usual way as
$\mathcal{B}_\nu^\pm=\sum_{j=1}^{n}\nu_j\gamma^\pm(\partial_j u)$,
where the unit normal $\nu$ points out of $\Omega^-$ and into
$\Omega^+$. We shall drop the $+$ or $-$ superscript if the
two-sided traces coincide. The adjoint operators $\gamma^*$ and
$\mathcal{B}_\nu^*$ are respectively defined by
\begin{equation}
(\gamma^*\psi,\phi)=(\psi,\gamma\phi)_\Pi\ \mbox{for \
$\phi\in\mathscr{E}(\mathbb{R}^n)$ and $\psi\in
H^{\epsilon-1}(\Pi)$, $0<\epsilon\leq 2$,}
\end{equation}
and
\begin{equation}
(\mathcal{B}_\nu^*\psi,\phi)=(\psi,\mathcal{B}_\nu\phi)_\Pi\quad\mbox{for\
$\phi\in\mathscr{E}(\mathbb{R}^n)$ and $\psi\in L^1(\Pi)$.}
\end{equation}
Let
\begin{equation}
\Phi(x,y)=\frac{e^{i\omega|x-y|}}{4\pi|x-y|}\ \ \mbox{for\ $n=3$}; \
\ \frac{i}{4}H_0^{(1)}(\omega|x-y|)\ \ \mbox{for\ $n=2$},
\end{equation}
with $H_0^{(1)}(t)$ the zeroth order Hankel function of the first
kind be the fundamental solutions to $-\Delta-\omega^2$,
respectively, in $\mathbb{R}^3$ and $\mathbb{R}^2$. Now, we define
\begin{equation}
\mathcal{G}u(x)=\int_{\mathbb{R}^n}\Phi(x,y)u(y)\,dy\ \ \mbox{for\
$x\in\mathbb{R}^n$},
\end{equation}
and the {\it single-layer potential} $SL$ and the {\it double-layer
potential} by
\begin{equation}
SL=\mathcal{G}\gamma*\quad\mbox{and}\quad
DL=\mathcal{G}\mathcal{B}_\nu^*,
\end{equation}
whose integral representations are given by
\begin{equation}
\begin{split}
SL\psi(x)=& \int_{\Pi}\Phi(x,y)\psi(y)\ dy,\\
DL\psi(x)=& \int_\Pi\frac{\partial\Phi(x,y)}{\partial\nu(y)}\psi(y)\
dy.
\end{split}
\end{equation}
In the sequel, for notational convenience, we shall write for a set
$\mathcal{W}\subset\mathbb{R}^n$,
\[
SL_{\Pi,\mathcal{W}}\psi(x):=\int_{\Pi}\Phi(x,y)\psi(y)\
dy\quad\mbox{for\ $y\in \mathcal{W}$},
\]
and similarly for $DL_{\Pi,\mathcal{W}}$. Let
\[
S_{\Pi}:=SL_{\Pi,\Pi},\quad D_{\Pi}:=DL_{\Pi,\Pi}
\]
which are understood in the sense of improper integrals. We shall
also need the {\it volume potentials}
\begin{align}
\mathcal{K}\psi(x):=&\int_{\mathbf{B}\backslash(\bar{U}_\rho\cup\bar{G})}\Phi(x,y)[1-q(y)]\psi(y)\
dy,\\
\widetilde{\mathcal{K}}\phi(x):=&\int_{\mathbf{B}\backslash\bar{G}}\Phi(x,y)[1-q(y)]\phi(y)\
dy,
\end{align}
and
\begin{equation}
\mathcal{J}\zeta(x):=\int_{U_\rho}\Phi(x,y)[q(y)-1]\zeta(y)\ dy
\end{equation}

We refer to \cite{ColKre,McL} for mapping properties of the
potential operators introduced here. Finally, we let
\begin{equation*}
\Omega=\mathbf{B}\backslash(\bar{U}_\rho\cup\bar{G}),\quad
\tilde{\Omega}=\mathbf{B}\backslash\bar{G},
\end{equation*}
and
\begin{align*}
&\partial G:=\Sigma;\ \ \Gamma_{k,\rho}:=\partial U_{k,\rho},\ \
\Gamma_\rho:=\bigcup_{k=1}^l\Gamma_{k,\rho}=\partial U_\rho;\\
&\ \ \ \ \ \Gamma_{k}:=\partial U_k,\
\Gamma:=\bigcup_{k=1}^l\Gamma_k=\partial U.
\end{align*}

\begin{lem}\label{lem:well posedness 1}
Let $v^i(x):=e^{ix\cdot\xi}$.
The solution $v\in
C^2(\mathbb{R}^n\backslash(\bar{U}_\rho\cup\bar{G}))\cap
C(\mathbb{R}^n\backslash({U}_\rho\cup {G}))$ to (\ref{eq:transformed
scattering}) is given by
\begin{equation}\label{eq:integral representation}
\begin{split}
v(x)=&v^i(x)
-\omega^2\int_{\mathbb{R}^n\backslash(\bar{U}_\rho\cup
\bar{G})}\Phi(x,y)[1-q(y)]v(y)\ dy\\
+&
\int_{\Sigma}\left[\frac{\partial\Phi(x,y)}{\partial\nu(y)}-i\Phi(x,y)\right]\psi_1(y)\
dy\\
+&\int_{\Gamma_\rho}\left[\frac{\partial\Phi(x,y)}{\partial\nu(y)}-i\kappa\Phi(x,y)\right]\psi_2(y)\
dy,
\end{split}
\end{equation}
where $x\in\mathbb{R}^n\backslash({U}_\rho\cup {G})$ and
$\kappa=[e(\rho)]^{-1}$. Here, $v|_{\bar{\Omega}}\in
C(\bar{\Omega})$, $\psi_1\in C(\Sigma)$ and $\psi_2\in
C(\Gamma_\rho)$ are uniquely determined by the following system of
integral equations,
\begin{align}
&v+\omega^2\mathcal{K}v-(DL_{\Sigma,\bar{\Omega}}-iSL_{\Sigma,\bar{\Omega}})\psi_1
\label{eq:sys1}\\ &\hspace*{1.6cm}-(DL_{\Gamma_\rho,\bar{\Omega}}
-i\kappa SL_{\Gamma_\rho,\bar{\Omega}})\psi_2=p\quad \mbox{in\ $\bar{\Omega}$},\nonumber\\
&\frac 1
2\psi_1-\omega^2\gamma\mathcal{K}v+(D_\Sigma-iS_\Sigma)\psi_1\label{eq:sys2}\\
&\hspace*{1.6cm}+(DL_{\Gamma_\rho,\Sigma}-i\kappa
SL_{\Gamma_\rho,\Sigma})\psi_2=q_1\quad\mbox{on\
$\Sigma$},\nonumber\\
&\frac 1 2
\psi_2-\omega^2\gamma\mathcal{K}v+(DL_{\Sigma,\Gamma_\rho}-iSL_{\Sigma,\Gamma_\rho})\psi_1\label{eq:sys3}\\
&\hspace*{1.6cm}+(D_{\Gamma_\rho}-i\kappa
S_{\Gamma_\rho})\psi_2=q_2\quad\mbox{on\ $\Gamma_\rho$},\nonumber
\end{align}
where $p(x)=v^i(x)$ for $x\in\bar{\Omega}$, $q_1(x)=-p(x)$ for
$x\in\Sigma$ and, $q_2(x)=-p(x)$ for $x\in\Gamma_\rho$. The system
(\ref{eq:sys1})--(\ref{eq:sys3}) is uniquely solvable in
$C(\bar{\Omega})\times C(\Sigma)\times C(\Gamma_\rho)$ for every
$p\in C(\bar{\Omega})$, $p_1\in C(\Sigma)$ and $p_2\in
C(\Gamma_\rho)$. It is also uniquely solvable in $L^2(\Omega)\times
C(\Sigma)\times C(\Gamma_\rho)$ for every $p\in L^2(\Omega)$ and,
$p_1\in C(\Sigma)$ and $p_2\in C(\Gamma_\rho)$.
\end{lem}

We remark that a similar scattering problem to (\ref{eq:transformed
scattering}) is considered in \cite{KirPai}, where the combination
of volume potential and boundary layer potentials are implemented to
represent the scattered wave field. However, the main concern of our
study is the scattering behaviors due to the asymptotically small
obstacles $U_\rho$. The novelty of the integral representation
(\ref{eq:integral representation}) is the introduction of the
asymptotically coupling parameter $\kappa$ for layer potentials on
the boundaries of the small obstacle components, which shall be
crucial for our subsequent scattering estimates. The proof of
Lemma~\ref{lem:well posedness 1} follows from a similar manner to
that in \cite{KirPai}, which for completeness we briefly include in
the following.

\begin{proof}[Proof of Lemma~\ref{lem:well posedness 1}] 
Using the mapping
properties of potential operators (cf. \cite{ColKre}), along with
the fact that $(-\Delta_y-\omega^2)\Phi(x,y)=\delta_x$, it is
straightforward to show that for $v$ given in (\ref{eq:integral
representation}), $v\in
C^2(\mathbb{R}^n\backslash(\bar{U}_\rho\cup\bar{G}))\cap
C(\mathbb{R}^n\backslash({U}_\rho\cup {G}))$ satisfies
$(\Delta+\omega^2q)v=0$. Whereas by the jump properties (cf.
\cite{ColKre}), (\ref{eq:sys2}) implies $v|_{\Sigma}=0$ and
(\ref{eq:sys3}) implies $v|_{\Gamma_\rho}=0$. The radiation
condition for $v^s$ is a direct consequence of the integral kernel
$\Phi(x,y)$. So, we only need to show the well-posedness of the
system (\ref{eq:sys1})--(\ref{eq:sys3}), which could be written as
\begin{equation}\label{eq:fredholm syste}
(\mathbf{A}+\mathbf{K})\mathbf{u}=\mathbf{p}
\end{equation}
where
\begin{equation}
\mathbf{A}=\left[
  \begin{array}{ccc}
    I & -DL_{\Sigma,\bar{\Omega}} & -DL_{\Gamma_\rho,\bar{\Omega}} \\
    \mathbf{0} & I & \mathbf{0} \\
    \mathbf{0} & \mathbf{0} & I \\
  \end{array}
\right],\ \mathbf{u}=\left(
         \begin{array}{c}
           v \\
           \psi_1 \\
           \psi_2 \\
         \end{array}
       \right),\ \mathbf{p}=\left(
         \begin{array}{c}
           p \\
           2q_1 \\
           2q_2 \\
         \end{array}
       \right)
\end{equation}
and
\begin{equation}
 \mathbf{K}=\left[
  \begin{array}{ccc}
    \omega^2\mathcal{K} & iSL_{\Sigma,\bar{\Omega}} & i\kappa SL_{\Gamma_\rho,\bar{\Omega}} \\
    -2\omega^2\gamma\mathcal{K} & 2(D_\Sigma-iS_\Sigma) & 2(DL_{\Gamma_\rho,\Sigma}-i\kappa SL_{\Gamma_\rho,\Sigma}) \\
    -2\omega^2\gamma\mathcal{K} & 2(DL_{\Sigma,\Gamma_\rho-iSL_{\Sigma,\Gamma_\rho}}) & 2(D_{\Gamma_\rho}-i\kappa S_{\Gamma_\rho}) \\
  \end{array}
\right].
\end{equation}
It is verified directly that
\[
\mathbf{A}^{-1}=\left[
  \begin{array}{ccc}
    I & DL_{\Sigma,\bar{\Omega}} & DL_{\Gamma_\rho,\bar{\Omega}} \\
    \mathbf{0} & I & \mathbf{0} \\
    \mathbf{0} & \mathbf{0} & I \\
  \end{array}
\right],
\]
and $\mathbf{K}$ is compact in both $L^2(\Omega)\times
C(\Sigma)\times C(\Gamma_\rho)$ and $C(\bar{\Omega})\times
C(\Sigma)\times C(\Gamma_\rho)$. So, $\mathbf{A}+\mathbf{K}$ is an
Fredholm operator of index 0. We only need to prove the uniqueness
of the system (\ref{eq:fredholm syste}), and it suffices to show
this in $L^2(\Omega)\times C(\Sigma)\times C(\Gamma_\rho)$. Set
$\mathbf{p}=0$. By the uniqueness of solution to the scattering
system (\ref{eq:transformed scattering}), we see $v=0$ in $\Omega$.
Let
\begin{align*}
\mathrm{w}(x)&=\int_{\Sigma}\left[\frac{\partial\Phi(x,y)}{\partial\nu(y)}-i\Phi(x,y)\right]\psi_1(y)
\ dy\\
+&
\int_{\Gamma_\rho}\left[\frac{\partial\Phi(x,y)}{\partial\nu(y)}-i\kappa\Phi(x,y)\right]\psi_2(y)\
dy,\quad x\in\mathbb{R}^n\backslash(\Sigma\cup\Gamma_\rho).
\end{align*}
Then, by (\ref{eq:sys1}), we see
$\mathrm{w}|_{\Omega}=v|_{\Omega}=0$. By the jump properties of
layer potential operators, we have
\[
\begin{split}
&-\gamma^-\mathrm{w}=\psi_1,\
-\gamma^-(\frac{\partial\mathrm{w}}{\partial\nu})=i\psi_1\ \
\mbox{on\ $\Sigma$};\\
&-\gamma^-\mathrm{w}=\psi_2,\
-\gamma^-(\frac{\partial\mathrm{w}}{\partial\nu})=i\kappa\psi_2\ \
\mbox{on\ $\Gamma_\rho$};
\end{split}
\]
By Green's formula,
\[
i\int_{\Sigma}|\psi_1|^2\
ds=\int_{\Sigma}(\gamma^-\mathrm{w})\gamma^-(\frac{\partial\mathrm{w}}{\partial\nu})\
ds=\int_{G} |\nabla \mathrm{w}|^2-\omega^2|\mathrm{w}|^2\ dx,
\]
which implies $\psi_1=0$. Similarly, one can show that $\psi_2=0$.
The proof is completed.

\end{proof}

\begin{rem}\label{rem:regular system}
Similar to Lemma~\ref{lem:well posedness 1}, 
we know the solution $\tilde v\in C^2(\mathbb{R}^n\backslash
\bar{G})\cap C(\mathbb{R}^n \backslash {G})$ to the scattering
problem without the small obstacle $U_\rho$ is given by
\begin{equation}\label{eq:integral representation 2}
\begin{split}
\tilde v(x)=& v^i(x)-\omega^2\int_{\mathbb{R}^n\backslash
\bar{G}}\Phi(x,y)[1-q(y)]\tilde{v}(y)\ dy\\
+&
\int_{\Sigma}\left[\frac{\partial\Phi(x,y)}{\partial\nu(y)}-i\Phi(x,y)\right]\tilde{\psi}(y)\
dy \quad x\in\mathbb{R}^n\backslash{G},
\end{split}
\end{equation}
with $\tilde{v}|_{\bar{\tilde{\Omega}}}\in C(\bar{\tilde{\Omega}})$
and $\tilde{\psi}\in C(\Sigma)$ uniquely determined by the following
system of integral equations,
\begin{align}
&\tilde{v}+\omega^2\tilde{\mathcal{K}}\tilde{v}-(DL_{\Sigma,\bar{\tilde\Omega}}-iSL_{\Sigma,\bar{\tilde\Omega}})\tilde\psi
=\tilde{p}\quad \mbox{in\ $\bar{\tilde\Omega}$},\label{eq:sys1 2}\\
&\frac 1
2\tilde{\psi}-\omega^2\gamma\tilde{\mathcal{K}}\tilde{v}+(D_\Sigma-iS_\Sigma)\tilde{\psi}=\tilde{q}\quad\mbox{on\
$\Sigma$},\label{eq:sys2 2}
\end{align}
where $\tilde{p}(x)=v^i(x)$ for $x\in\bar{\Omega}$,
$\tilde{q}(x)=-\tilde{p}(x)$ for $x\in\Sigma$. For the subsequent
use, we let
\begin{equation}
\begin{split}
& \tilde{\mathbf{L}}=\left[
             \begin{array}{cc}
               I+\omega^2\tilde{\mathcal{K}} & -DL_{\Sigma,\bar{\tilde{\Omega}}}+iSL_{\Sigma,\bar{\tilde{\Omega}}} \\
               -\omega^2\gamma\tilde{\mathcal{K}} & \frac 1 2 I+D_\Sigma-iS_\Sigma \\
             \end{array}
           \right],\\
           & \qquad \tilde{\mathbf{v}}=\left[
                                              \begin{array}{c}
                                                \tilde{v} \\
                                                \tilde{\psi} \\
                                              \end{array}
                                            \right],\quad \tilde{\mathbf{x}}=\left[
                                              \begin{array}{c}
                                                \tilde{p} \\
                                                \tilde{q} \\
                                              \end{array}
                                            \right].
\end{split}
\end{equation}
Clearly, we have
\begin{equation}
\tilde{\mathbf{v}}=\tilde{\mathbf{L}}^{-1}\tilde{\mathbf{x}},
\end{equation}
both in $L^2(\tilde{\Omega})\times C(\Sigma)$ and
$C(\bar{\tilde{\Omega}})\times C(\Sigma)$.
\end{rem}

%
%

We next derive the key lemma in proving Theorem~\ref{thm:small}. In
the following, the analysis is based on the space dimension being 3.
Later, we shall indicate the necessary modifications for the two
dimensional case. Henceforth, for an operator $\Lambda: X\mapsto Y$
with $X, Y$ being Banach spaces, we shall denote by
$\|\Lambda\|_{\mathcal{L}(X,Y)}$ the corresponding operator norm.

\begin{lem}\label{lem:key lemma}
For $\rho\rightarrow 0^+$, we have
\begin{equation}
\|v-\tilde{v}\|_{L^2(\Omega)}=\|v-\tilde{v}\|_{L^2(\tilde{\Omega}\backslash
U_\rho)}=\mathcal{O}(\rho)
\end{equation}
and
\begin{equation}
\|\psi_1-\tilde\psi\|_{C(\Sigma)}=\mathcal{O}(\rho)\quad\mbox{and}\quad
\|\psi_2\|_{C(\Gamma_\rho)}=\mathcal{O}(1).
\end{equation}
\end{lem}

\begin{proof}
The proof shall be proceeded in 6 steps.\\

\noindent{\bf Step I.}~~We introduce
$w=w_1\chi_{\tilde{\Omega}\backslash U_\rho}+w_2\chi_{U_\rho}\in
L^2(\tilde\Omega)$ such that
\begin{equation}\label{eq:k 2}
w_1=v\quad\mbox{in \ $\Omega=\tilde\Omega\backslash U_\rho$},
\end{equation}
and
\begin{equation}\label{eq:k 1}
\begin{split}
w_2(x)+&\omega^2\int_{\tilde{\Omega}\backslash
U_\rho}\Phi(x,y)[1-q(y)]w_1(y)\,dy\\
-&(DL_{\Sigma,U_\rho}-iSL_{\Sigma,U_\rho})\psi_1(x)\\
-&(DL_{\Gamma_\rho,U_\rho} -i\kappa
SL_{\Gamma_\rho,U_\rho})\psi_2(x)=p(x)\quad x\in U_\rho.
\end{split}
\end{equation}
By (\ref{eq:k 1}), (\ref{eq:k 2}) and (\ref{eq:sys1}), one can see
that $w\in L^2(\tilde{\Omega})$ and $w_2\in L^2(U_\rho)$ satisfy the
following system of operator equations
\begin{align}
w+\omega^2\widetilde{\mathcal{K}}w+\omega^2\mathcal{J}w_2-&(DL_{\Sigma,\tilde\Omega}-iSL_{\Sigma,\tilde\Omega})\psi_1\nonumber\\
-&(DL_{\Gamma_\rho,\tilde\Omega} -i\kappa
SL_{\Gamma_\rho,\tilde\Omega})\psi_2=p\quad\mbox{in\ \
$\tilde\Omega$},\label{eq:w}\\
w_2+\omega^2\mathcal{J}w_2+\omega^2\widetilde{\mathcal{K}}w-&(DL_{\Sigma,U_\rho}-iSL_{\Sigma,U_\rho})\psi_1\nonumber\\
-&(DL_{\Gamma_\rho,U_\rho}-i\kappa
SL_{\Gamma_\rho,U_\rho})\psi_2=p\quad\mbox{in\ \
$U_\rho$}.\label{eq:w2}
\end{align}

\noindent{\bf Step II.}~Let $\Theta(x)=x/\rho:\
U_{\rho}\cup\Gamma_\rho\rightarrow U\cup\Gamma$. In the sequel, for
$x\in \Gamma$, we define
\[
(S_\Gamma^0\phi)(x)=\int_{\Gamma}\Phi_0(x,y)\phi(y)\ dy,\quad
(D_\Gamma^0\phi)(x)=\int_{\Gamma}\frac{\Phi_0(x,y)}{\partial\nu(y)}\phi(y)\
dy,
\]
where $\Phi_0(x,y)=1/(4\pi|x-y|)$ is the fundamental solution to
$-\Delta$. For $\phi\in C(\Gamma_\rho)$, by using change of
variables in integration, one has
\begin{align*}
& \int_{\Gamma_\rho}\frac{e^{i\omega|x-y|}}{|x-y|}\phi(y)\
dy=\rho\int_{\Gamma}\frac{e^{i\omega\rho|x'-y'|}}{|x'-y'|}\phi(\rho
y')\ dy',\\
&\int_{\Gamma_\rho}\partial\left(\frac{e^{i\omega|x-y|}}{|x-y|}\right)/\partial\nu(y)\phi(y)\
dy=\int_{\Gamma}\partial\left(\frac{e^{i\omega\rho|x'-y'|}}{|x'-y'|}\right)/\partial\nu(y')\phi(\rho
y')\ dy',
\end{align*}
where $x'=x/\rho, y'=y/\rho\in\Gamma$. Using these along with power
series expansion of $\exp\{i\omega\rho|x'-y'|\}$, one can verify
directly that for $\phi\in C(\Gamma_\rho)$
\begin{equation}\label{eq:property1}
\|(\Theta^{-1})^*\circ
S_{\Gamma_\rho}\circ(\Theta^{-1})^*\phi-S_\Gamma^0\circ(\Theta^{-1})^*\phi\|_{C(\Gamma)}\lesssim\rho
\|\phi\|_{C(\Gamma_\rho)},
\end{equation}
where in the above inequality we have identified $\phi(x)\in
C(\Gamma_\rho)$ as $\phi(x)=\phi(\rho x')=(\Theta^{-1})^*\phi\in
C(\Gamma)$, and the fact that
$\|\phi\|_{C(\Gamma_\rho)}=\|(\Theta^{-1})^*\phi\|_{C(\Gamma)}$.
Similarly, we have
\begin{equation}\label{eq:property2}
\|(\Theta^{-1})^*\circ
D_{\Gamma_\rho}\circ(\Theta^{-1})^*\phi-D_\Gamma^0\circ(\Theta^{-1})^*\phi\|_{C(\Gamma)}\lesssim\rho^2
\|\phi\|_{C(\Gamma_\rho)}.
\end{equation}
Next, by (\ref{eq:sys3}), together with the use of
(\ref{eq:property1}) and (\ref{eq:property2}), we have
\begin{equation}\label{eq:psi2}
\begin{split}
(\Theta^{-1})^*\psi_2=& [\frac 1
2I+D_\Gamma^0-iS_\Gamma^0+\mathcal{O}(\rho)]^{-1}(\Theta^{-1})^*\bigg[q_2+\omega^2\gamma\mathcal{K}v\\
&-(DL_{\Sigma,\Gamma_\rho}-iSL_{\Sigma,\Gamma_\rho})\psi_1\bigg].
\end{split}
\end{equation}
Here, we made use of the invertibility of $(\frac 1
2I+D_\Gamma^0-iS_\Gamma^0):\ C(\Gamma)\mapsto C(\Gamma)$ (cf.
\cite{ColKre2}).\\

\noindent{\bf Step III.}~Using (\ref{eq:psi2}), we consider
\begin{equation}\label{eq:A0}
\begin{split}
&(DL_{\Gamma_\rho,\Sigma}-i\kappa SL_{\Gamma_\rho,\Sigma})\psi_2\\
=& (DL_{\Gamma_\rho,\Sigma}-i\kappa
SL_{\Gamma_\rho,\Sigma})\Theta^*\bigg\{[\frac 1
2I+D_\Gamma^0-iS_\Gamma^0+\mathcal{O}(\rho)]^{-1}\\
&\qquad (\Theta^{-1})^*\bigg[q_2
+\omega^2\gamma\mathcal{K}v-(DL_{\Sigma,\Gamma_\rho}-iSL_{\Sigma,\Gamma_\rho})\psi_1\bigg]\bigg\}
\end{split}
\end{equation}
We first note that
\[
\mbox{dist}_{\mathbb{R}^n}(\Gamma,\Sigma)<\mbox{dist}_{\mathbb{R}^n}(\Gamma_\rho,\Sigma)<\mbox{diam}_{\mathbb{R}^n}(\mathbf{B}),
\]
and hence by direct verification
\begin{equation}\label{eq:A1}
\|DL_{\Gamma_\rho,\Sigma}-i\kappa
SL_{\Gamma_\rho,\Sigma}\|_{\mathcal{L}(C(\Gamma_\rho),
C(\Sigma))}=\mathcal{O}(\rho)
\end{equation}
and
\begin{equation}\label{eq:A2}
\|DL_{\Sigma,\Gamma_\rho}-iSL_{\Sigma,\Gamma_\rho}\|_{\mathcal{L}(C(\Sigma),
C(\Gamma_\rho))}=\mathcal{O}(1)).
\end{equation}
By letting $\hat{v}=v\chi_{\Omega}+0\chi_{U_\rho}\in
L^2(\tilde{\Omega})$, and using the fact that
$\widetilde{\mathcal{K}}$ maps $L^2(\tilde{\Omega})$ continuously
into $H^2(\tilde{\Omega})$ (cf. \cite{ColKre}), we also see that
\begin{equation}\label{eq:A3}
\begin{split}
&\|\mathcal{K}v\|_{C(\Omega)}\leq\|\widetilde{\mathcal{K}}\hat{v}\|_{C(\tilde{\Omega})}
\lesssim
\|\widetilde{\mathcal{K}}\hat{v}\|_{H^2(\tilde{\Omega})}\\
&\lesssim\|\hat{v}\|_{L^2(\tilde\Omega)} =\|v\|_{L^2(\Omega)}\leq
\|w\|_{L^2(\tilde{\Omega})}.
\end{split}
\end{equation}
By (\ref{eq:A0})--(\ref{eq:A3}), we have
\begin{equation}
(DL_{\Gamma_\rho,\Sigma}-i\kappa
SL_{\Gamma_\rho,\Sigma})\psi_2=\mathcal{A}_1(v)+\mathcal{A}_2(\psi_1)+\mathcal{A}_3(q_2),
\end{equation}
where
\begin{equation}
\begin{split}
\mathcal{A}_1(v)=& (DL_{\Gamma_\rho,\Sigma}-i\kappa
SL_{\Gamma_\rho,\Sigma})\Theta^*\bigg([\frac 1
2I+D_\Gamma^0-iS_\Gamma^0\\
&\ \
+\mathcal{O}(\rho)]^{-1}(\Theta^{-1})^*(\omega^2\gamma\mathcal{K}v)\bigg)
\end{split}
\end{equation}
satisfying
\begin{equation}\label{eq:F1}
\|\mathcal{A}_1(v)\|_{C(\Sigma)}\lesssim
\rho\|v\|_{L^2(\Omega)}\lesssim\rho\|w\|_{L^2(\tilde{\Omega})};
\end{equation}
and
\begin{equation}
\begin{split}
\mathcal{A}_2(\psi_1)=& (DL_{\Gamma_\rho,\Sigma}-i\kappa
SL_{\Gamma_\rho,\Sigma})\Theta^*\bigg([\frac 1
2I+D_\Gamma^0-iS_\Gamma^0+\mathcal{O}(\rho)]^{-1}\\
&(\Theta^{-1})^*(-DL_{\Sigma,\Gamma_\rho}+iSL_{\Sigma,\Gamma_\rho})\psi_1\bigg)
\end{split}
\end{equation}
satisfying
\begin{equation}
\|\mathcal{A}_2(\psi_1)\|\lesssim\rho\|\psi_1\|_{C(\Sigma)};
\end{equation}
and
\begin{equation}
\begin{split}
\mathcal{A}_3(q_2)=& (DL_{\Gamma_\rho,\Sigma}-i\kappa
SL_{\Gamma_\rho,\Sigma})\Theta^*\\
&\ \ ([\frac 1
2I+D_\Gamma^0-iS_\Gamma^0+\mathcal{O}(\rho)]^{-1}(\Theta^{-1})^*q_2)
\end{split}
\end{equation}
satisfying
\begin{equation}\label{eq:B0}
\|\mathcal{A}_3(q_2)\|_{C(\Sigma)}\lesssim
\rho\|q_2\|_{C(\Gamma_\rho)}.
\end{equation}
Hence, we
see from (\ref{eq:B0}) that
\begin{equation}\label{eq:B1}
\|\mathcal{A}_3(q_2)\|_{C(\Sigma)}\lesssim \rho.
\end{equation}

Plugging (\ref{eq:psi2}) into (\ref{eq:sys2}), and using
(\ref{eq:A0})--(\ref{eq:B1}), we have
\begin{equation}\label{eq:psi11}
[\frac 1 2
I+D_\Sigma-iS_\Sigma+\mathcal{O}(\rho)]\psi_1-\omega^2\gamma\mathcal{K}v+\mathcal{A}_1(v)=\tilde{q}+\mathcal{O}(\rho)\quad\mbox{on
\ $\Sigma$}.
\end{equation}

\medskip

\noindent{\bf Step IV.}~~ In the sequel, we shall denote by $DL^0$
and $SL^0$ the double- and single-layer potentials with the integral
kernel $\Phi(x,y)$ replaced by $\Phi_0(x,y)$. Also, we let
\[
\mathcal{J}_0\zeta(x)=\int_{U}\Phi_0(x,y)[1-\hat{q}(y)]\zeta(y)\ dy,
\]
where $\hat{q}=(\Theta^{-1})^*q\in L^\infty(U)$. We know that
$I+\mathcal{J}_0$ is bounded invertible from $L^2(U)$ to itself and
(cf. \cite{ColKre})
\begin{equation}\label{eq:bound}
\|(I+\omega^2\mathcal{J}_0)^{-1}\|_{\mathcal{L}(L^2(U), L^2(U))}\leq
c(U,\|\hat{q}\|_{L^\infty(U)},\omega).
\end{equation}
Clearly, the bound in (\ref{eq:bound}) is independent of $\rho$.
Now, using change of variables in integration and by a similar
argument to that in Step II, we have from (\ref{eq:w2})
\begin{equation}\label{eq:w22}
\begin{split}
(\Theta^{-1})^*&w_2=[I+\omega^2\mathcal{J}_0+\mathcal{O}(\rho)]^{-1}(\Theta^{-1})^*\bigg[-\omega^2\widetilde{\mathcal{K}}w\\
+&
(DL_{\Sigma,U_\rho}-iSL_{\Sigma,U_\rho})\psi_1+(DL_{\Gamma_\rho,U_\rho}-i\kappa
SL_{\Gamma_\rho,U_\rho})\psi_2+p\bigg]\\
=&
\mathcal{M}_1(w)+\mathcal{M}_2(\psi_1)+\mathcal{M}_3(\psi_2)+\mathcal{M}_4(p).
\end{split}
\end{equation}
We next assess the $L^2(U)$-norms of $\mathcal{M}_1(w)$,
$\mathcal{M}_2(\psi_1)$, $\mathcal{M}_3(\psi_2)$ and
$\mathcal{M}_4(p)$, respectively. In (\ref{eq:w22}),
\[
\mathcal{M}_1(w)=[I+\omega^2\mathcal{J}_0+\mathcal{O}(\rho)]^{-1}(\Theta^{-1})^*[-\omega^2\widetilde{\mathcal{K}}w],
\]
which satisfies
\begin{equation}
\|\mathcal{M}_1(w)\|_{L^2(U)}\lesssim\|\widetilde{\mathcal{K}}w\|_{C(\tilde\Omega)}\lesssim\|w\|_{L^2(\tilde{\Omega})}.
\end{equation}
In (\ref{eq:w22}).
\[
\mathcal{M}_2(\psi_1)=[I+\omega^2\mathcal{J}_0+\mathcal{O}(\rho)]^{-1}(\Theta^{-1})^*[(DL_{\Sigma,U_\rho}-iSL_{\Sigma,U_\rho})\psi_1],
\]
Noting
\[
\mbox{dist}_{\mathbb{R}^n}(\Sigma,\Gamma)<\mbox{dist}_{\mathbb{R}^n}(\Sigma,\bar{U}_\rho)<\mbox{diam}_{\mathbb{R}^n}(\mathbf{B}),
\]
one can easily show that
\[
\|(DL_{\Sigma,U_\rho}-iSL_{\Sigma,U_\rho})\psi_1\|_{C(U_\rho)}\lesssim\|\psi_1\|_{C(\Sigma)},
\]
and hence
\begin{equation}
\|\mathcal{M}_2(\psi_1)\|_{L^2(U)}\lesssim\|\psi_1\|_{C(\Sigma)}.
\end{equation}
For $\mathcal{M}_3(\psi_2)$, we have
\[
\mathcal{M}_3(\psi_2)=[I+\omega^2\mathcal{J}_0+\mathcal{O}(\rho)]^{-1}(\Theta^{-1})^*[(DL_{\Gamma_\rho,U_\rho}-i\kappa
SL_{\Gamma_\rho,U_\rho})\psi_2].
\]
It is verified directly that
\begin{align*}
(DL_{\Gamma_\rho,U_\rho}\psi_2)(x)=&\int_{\Gamma_\rho}\partial\left(\frac{e^{i\omega|x-y|}}{|x-y|}\right)/\partial\nu(y)\psi_2(y)\
dy\\
=&\int_{\Gamma}\partial\left(\frac{e^{i\omega\rho|x'-y'|}}{|x'-y'|}\right)/\partial\nu(y')\psi_2(\rho
y')\ dy',\\
(SL_{\Gamma_\rho,U_\rho}\psi_2)(x)=&\int_{\Gamma_\rho}\frac{e^{i\omega|x-y|}}{|x-y|}\psi_2(y)\
dy=\rho\int_{\Gamma}\frac{e^{i\omega\rho|x'-y'|}}{|x'-y'|}\psi_2(\rho
y')\ dy',
\end{align*}
where $x\in\Gamma_\rho, y\in U_\rho$ and $x'\in\Gamma, y'\in U$.
Hence, we have
\[
[DL_{\Gamma_\rho,U_\rho}-i\kappa
SL_{\Gamma_\rho,U_\rho}]\psi_2=[DL_{\Gamma,U}^0+iSL_{\Gamma,U}^0+\mathcal{O}(\rho)](\Theta^{-1})^*\psi_2,
\]
and from which we further have
\begin{equation}
\|\mathcal{M}_3(\psi_2)\|_{L^2(U)}\lesssim
\|\psi_2\|_{C(\Gamma_\rho)}.
\end{equation}
Finally, by
\[
\mathcal{M}_4(p)=[I+\omega^2\mathcal{J}_0+\mathcal{O}(\rho)]^{-1}(\Theta^{-1})^*p,
\]
we see
\begin{equation}\label{eq:E0}
\|\mathcal{M}_4(p)\|_{L^2(U)}=\mathcal{O}(1).
\end{equation}
Now, by (\ref{eq:w22})--(\ref{eq:E0}), we have
\begin{equation}\label{eq:C0}
\begin{split}
\|w_2\|&_{L^2(U_\rho)}=\rho^{3/2}\|(\Theta^{-1})^*w_2\|_{L^2(U)}\leq\rho^{3/2}\bigg(\|\mathcal{M}_1(w)\|_{L^2(U)}\\
&
+\|\mathcal{M}_2(\psi_1)\|_{L^2(U)}+\|\mathcal{M}_3(\psi_2)\|_{L^2(U)}+\|\mathcal{M}_4(p)\|_{L^2(U)}\bigg)\\
\lesssim&\
\rho^{3/2}\|w\|_{L^2(\tilde\Omega)}+\rho^{3/2}\|\psi_1\|_{C(\Sigma)}+\rho^{3/2}\|\psi_2\|_{C(\Gamma_\rho)}+\mathcal{O}(\rho^{3/2}).
\end{split}
\end{equation}

\medskip

\noindent{\bf Step V.}~~We now consider the equation (\ref{eq:w}).
Let $\hat{w}=w_2\chi_{U_\rho}+0\chi_{\Omega}\in L^2(\tilde\Omega)$.
Using (\ref{eq:C0}), we have
\begin{equation}\label{eq:C2}
\begin{split}
&\|\mathcal{J}w_2\|_{H^2(\tilde\Omega)}=\|-\widetilde{\mathcal{K}}\hat{w}\|_{H^2(\tilde\Omega)}\lesssim\|\hat{w}\|_{L^2(\tilde\Omega)}
=\|w_2\|_{L^2(U_\rho)}\\
\lesssim&\
\rho^{3/2}\|w\|_{L^2(\tilde\Omega)}+\rho^{3/2}\|\psi_1\|_{C(\Sigma)}+\rho^{3/2}\|\psi_2\|_{C(\Gamma_\rho)}+\mathcal{O}(\rho^{3/2}).
\end{split}
\end{equation}
By (\ref{eq:psi2}), (\ref{eq:A2}) and (\ref{eq:A3}), one can show
\begin{equation}\label{eq:E3}
\begin{split}
&\|\psi_2\|_{C(\Gamma_\rho)}=\|(\Theta^{-1})^*\psi_2\|_{C(\Gamma)}\\
\lesssim&\
\|q_2\|_{C(\Gamma_\rho)}+\|w\|_{L^2(\tilde\Omega)}+\|\psi_1\|_{C(\Sigma)}\\
\lesssim&\
\|\tilde{q}\|_{C(\tilde{\Omega})}+\|w\|_{L^2(\tilde\Omega)}+\|\psi_1\|_{C(\Sigma)}+\mathcal{O}(\rho),
\end{split}
\end{equation}
which together with (\ref{eq:C0}) and (\ref{eq:C2}) implies
\begin{equation}\label{eq:C1}
\|w_2\|_{L^2(U_\rho)}\lesssim\
\rho^{3/2}\|w\|_{L^2(\tilde\Omega)}+\rho^{3/2}\|\psi_1\|_{C(\Sigma)}+\mathcal{O}(\rho^{3/2}),
\end{equation}
and
\begin{equation}\label{eq:C3}
\|\mathcal{J}w_2\|_{H^2(\tilde{\Omega})}\lesssim\
\rho^{3/2}\|w\|_{L^2(\tilde\Omega)}+\rho^{3/2}\|\psi_1\|_{C(\Sigma)}+\mathcal{O}(\rho^{3/2}).
\end{equation}
Let $\vartheta\in\mathscr{D}(\mathbb{R}^3)$ be a cut-off function
satisfying $\vartheta=1$ on $\mathbf{B}$ and $\vartheta=0$ on
$\mathbf{D}$ with $\mathbf{D}$ a neighborhood of $\mathbf{B}$. For
$\phi\in\mathscr{D}(\mathbb{R}^n)$, we have
\begin{equation}
\begin{split}
&
(DL_{\Gamma_\rho,\mathbb{R}^n}\psi_2,\vartheta\phi)_{L^2(\mathbb{R}^n)}
=(\mathcal{G}\vartheta\mathcal{B}_\nu^*\psi_2,\vartheta\phi)_{L^2(\mathbb{R}^n)}\\
=&(\psi_2,\mathcal{B}_\nu(\vartheta\mathcal{G}^*\vartheta)\phi)_{L^2(\Gamma_\rho)}
\leq
\|\psi_2\|_{L^2(\Gamma_\rho)}\|\mathcal{B}_\nu(\vartheta\mathcal{G}^*\vartheta)\phi\|_{L^2(\Gamma_\rho)}\\
\lesssim&\
\rho\|\psi_2\|_{C(\Gamma_\rho)}\|(\vartheta\mathcal{G}^*\vartheta)\phi\|_{H^2(\mathbf{D})}\\
\lesssim&\
\rho\|\psi_2\|_{C(\Gamma_\rho)}\|\vartheta\phi\|_{L^2(\mathbb{R}^n)},
\end{split}
\end{equation}
and from which one easily see that
\begin{equation}\label{eq:E1}
\|DL_{\Gamma_\rho,\tilde{\Omega}}\psi_2\|_{L^2(\tilde{\Omega})}\lesssim\rho\|\psi_2\|_{C(\Gamma_\rho)}.
\end{equation}
Here, we have made use of the fact that $\mathcal{G}^*$ maps
$L^2(\mathbf{D})$ continuously into $H^2(\mathbf{D})$. Similarly, we
have
\begin{equation}
\begin{split}
&
(SL_{\Gamma_\rho,\mathbb{R}^n}\psi_2,\vartheta\phi)_{L^2(\mathbb{R}^n)}
=(\mathcal{G}\vartheta\gamma^*\psi_2,\vartheta\phi)_{L^2(\mathbb{R}^n)}\\
=&(\psi_2,\gamma(\vartheta\mathcal{G}^*\vartheta)\phi)_{L^2(\Gamma_\rho)}
\leq
\|\psi_2\|_{L^2(\Gamma_\rho)}\|\gamma(\vartheta\mathcal{G}^*\vartheta)\phi\|_{L^2(\Gamma_\rho)}\\
\lesssim&\
\rho^2\|\psi_2\|_{C(\Gamma_\rho)}\|(\vartheta\mathcal{G}^*\vartheta)\phi\|_{C(\Gamma_\rho)}\\
\lesssim&\
\rho^2\|\psi_2\|_{C(\Gamma_\rho)}\|(\vartheta\mathcal{G}^*\vartheta)\phi\|_{H^2(\mathbf{D})}\\
\lesssim&\
\rho^2\|\psi_2\|_{C(\Gamma_\rho)}\|\vartheta\phi\|_{L^2(\mathbb{R}^n)},
\end{split}
\end{equation}
and from which one has
\begin{equation}\label{eq:E2}
\|\kappa
SL_{\Gamma_\rho,\tilde{\Omega}}\psi_2\|_{L^2(\tilde{\Omega})}=\rho^{-1}\|
SL_{\Gamma_\rho,\tilde{\Omega}}\psi_2\|_{L^2(\tilde{\Omega})}\lesssim\rho\|\psi_2\|_{C(\Gamma_\rho)}.
\end{equation}
By (\ref{eq:E1}) and (\ref{eq:E2}), we see that
\begin{equation}
\|(DL_{\Gamma_\rho,\tilde\Omega}-i\kappa
SL_{\Gamma_\rho,\tilde\Omega})\psi_2\|_{L^2(\tilde\Omega)}\lesssim
\rho\|\psi_2\|_{C(\Gamma_\rho)},
\end{equation}
which together with (\ref{eq:E3}) further gives
\begin{equation}\label{eq:E4}
\begin{split}
&\|(DL_{\Gamma_\rho,\tilde\Omega}-i\kappa
SL_{\Gamma_\rho,\tilde\Omega})\psi_2\|_{L^2(\tilde\Omega)}\\
\lesssim & \
\rho\|w\|_{L^2(\tilde\Omega)}+\rho\|\psi_1\|_{C(\Sigma)}+\mathcal{O}(\rho).
\end{split}
\end{equation}
Obviously, (\ref{eq:C3}) implies that
\begin{equation}\label{eq:E5}
\begin{split}
&\|\mathcal{J}w_2\|_{L^2(\tilde\Omega)}\leq
\|\mathcal{J}w_2\|_{H^2(\tilde{\Omega})}\\
\lesssim & \
\rho^{3/2}\|w\|_{L^2(\tilde\Omega)}+\rho^{3/2}\|\psi_1\|_{C(\Sigma)}+\mathcal{O}(\rho^{3/2}).
\end{split}
\end{equation}
Finally, (\ref{eq:w}), (\ref{eq:E4}) and (\ref{eq:E5}), together
with the fact that $q=\tilde{q}$ in
$\tilde{\Omega}$, yield that
$(w,\psi_1)\in L^2(\tilde{\Omega})\times C(\Sigma)$ satisfies
\begin{equation}\label{eq:w23}
\begin{split}
&[I+\omega^2\widetilde{\mathcal{K}}+\mathcal{O}(\rho)]w-[DL_{\Sigma,\tilde\Omega}-
iSL_{\Sigma,\tilde\Omega}+\mathcal{O}(\rho)]\psi_1\\
=& \tilde{p}+\mathcal{O}(\rho)\quad\mbox{in\ $\tilde\Omega$}.
\end{split}
\end{equation}

\medskip

\noindent{\bf Step VI.}~~In the following, we further assess
(\ref{eq:psi11}). We first note
\begin{equation}\label{eq:F0}
\gamma\mathcal{K}v=\gamma\widetilde{\mathcal{K}}w+\gamma\mathcal{J}w_2\quad\mbox{on\
$\Sigma$}.
\end{equation}
By (\ref{eq:C2}) and (\ref{eq:E3}), one can show
\begin{equation}
\begin{split}
&\|\gamma\mathcal{J}w_2\|_{C(\Sigma)}=\|-\gamma\widetilde{\mathcal{K}}\hat{w}\|_{C(\Sigma)}
\lesssim\|\widetilde{\mathcal{K}}\hat{w}\|_{H^2(\tilde\Omega)}\\
\lesssim&\
\rho^{3/2}\|w\|_{L^2(\tilde\Omega)}+\rho^{3/2}\|\psi_1\|_{C(\Sigma)}+\mathcal{O}(\rho^{3/2}),
\end{split}
\end{equation}
which together with (\ref{eq:F0}), (\ref{eq:psi11}) and
(\ref{eq:F1}) implies
\begin{equation}\label{eq:psi12}
[\frac 1 2
I+D_\Sigma-iS_\Sigma+\mathcal{O}(\rho)]\psi_1-[\omega^2\gamma\mathcal{K}+\mathcal{O}(\rho)]w=\tilde{q}+\mathcal{O}(\rho)\quad\mbox{on
\ $\Sigma$}.
\end{equation}
Now, by (\ref{eq:w23}) and (\ref{eq:psi12}) we see that
$\mathbf{v}:=(w,\psi_1)\in L^2(\tilde{\Omega})\times C(\Sigma)$
satisfies
\begin{equation}
[\tilde{\mathbf{L}}+\mathcal{O}(\rho)]\mathbf{v}=\tilde{\mathbf{x}}+\mathcal{O}(\rho),
\end{equation}
which by Remark~\ref{rem:regular system} gives
\begin{equation}
\mathbf{v}=\tilde{\mathbf{v}}+\mathcal{O}(\rho).
\end{equation}
That is,
\begin{equation}
\|w-\tilde{v}\|_{L^2(\tilde\Omega)}=\mathcal{O}(\rho)\quad\mbox{and}\quad
\|\psi_1-\tilde{\psi}\|_{C(\Sigma)}=\mathcal{O}(\rho),
\end{equation}
which together with (\ref{eq:E3}) implies
\begin{equation}
\|\psi_2\|_{C(\Gamma_\rho)}=\mathcal{O}(1).
\end{equation}
It is trivially pointed out that
\begin{equation}
\|v-\tilde{v}\|_{L^2(\Omega)}\leq\|w-\tilde{v}\|_{L^2(\tilde\Omega)}=\mathcal{O}(\rho).
\end{equation}

The proof is completed.

\end{proof}

\begin{rem}\label{rem:2d case}
In two dimensions, we shall have as $\rho\rightarrow 0^+$
\begin{equation}
\|v-\tilde{v}\|_{L^2(\Omega)}=\mathcal{O}((\ln\rho)^{-1})
\end{equation}
and
\begin{equation}
\|\psi_1-\tilde\psi\|_{C(\Sigma)}=\mathcal{O}((\ln\rho)^{-1})\quad\mbox{and}\quad
\|\psi_2\|_{C(\Gamma_\rho)}=\mathcal{O}(1).
\end{equation}
The proof is very similar to that for the three dimensional case in
Lemma~\ref{lem:key lemma}, and the only modification we shall need
is that at certain point, we shall make use of the following
asymptotic expansion for $\Phi(x,y)$,
\[
\begin{split}
\Phi(x,y)=&\frac{i}{4}H_0^{(1)}(\omega|x-y|)=\frac{1}{2\pi}\ln\frac{1}{|x-y|}+\frac{i}{4}\\
&-\frac{1}{2\pi}\ln\frac{\omega}{2}-\frac{E}{2\pi}
+\mathcal{O}\bigg(|x-y|^2\ln\frac{1}{|x-y|}\bigg)
\end{split}
\]
as $|x-y|\rightarrow 0$. Here, $E$ denotes the Euler's constant.
\end{rem}

\begin{proof}[Proof of Theorem~\ref{thm:small}]
By noting $u^s=v^s$ in $\mathbb{R}^n\backslash\bar{\mathbf{B}}$ and (\ref{eq:integral representation}),
one has by direct asymptotic expansion that
\begin{equation}\label{eq:farfield integral representation}
\begin{split}
A(\theta,\xi;\mathcal{T})=-& \omega^2\tau_0\int_{\Omega}e^{-i\omega\theta\cdot y}[1-q(y)]v(y)\ dy\\
+& \tau_0\int_{\Sigma}\left[\frac{\partial e^{-i\omega\theta\cdot
y}}{\partial\nu(y)}-ie^{-i\omega\theta\cdot y}\right]\psi_1(y)\
dy\\
+&\tau_0\int_{\Gamma_\rho}\left[\frac{\partial
e^{-i\omega\theta\cdot y}}{\partial\nu(y)}-i\kappa
e^{-i\omega\theta\cdot y}\right]\psi_2(y)\ dy,
\end{split}
\end{equation}
where $\tau_0=e^{i\frac{\pi}{4}}/\sqrt{8\pi\omega}$ for $n=2$; and
$\tau_0=\frac{1}{4\pi}$ for $n=3$. Similarly, by
$\tilde{u}^s=\sigma_0^{-1/2}\tilde{v}^s$ and (\ref{eq:integral
representation 2}), one also has
\begin{equation}\label{eq:farfield integral representation 2}
\begin{split}
& A(\theta,\xi;G\oplus\{\mathbf{B}\backslash\bar{G};\sigma_0,\eta_0\})\\
=&\ 
-\omega^2\tau_0\int_{\tilde{\Omega}}e^{-i\omega\theta\cdot
y}[1-q(y)]\tilde{v}(y)\ dy\\
&\ +\tau_0\int_{\Sigma}\left[\frac{\partial e^{-i\omega\theta\cdot
y}}{\partial\nu(y)}-ie^{-i\omega\theta\cdot
y}\right]\tilde{\psi}(y)\ dy
\end{split}
\end{equation}
Subtracting (\ref{eq:farfield integral representation 2}) from
(\ref{eq:farfield integral representation}), we have
\begin{equation}\label{eq:farfield difference}
\begin{split}
&A(\theta,\xi;\mathcal{T})-A(\theta,\xi;G\oplus\{\mathbf{B}\backslash\bar{G};\sigma_0,\eta_0\})\\
=& \ -\omega^2\tau_0\int_{\Omega}e^{-i\omega\theta\cdot y}[1-q(y)][v(y)-\tilde{v}(y)]\ dy\\
& + \tau_0\int_{\Sigma}\left[\frac{\partial e^{-i\omega\theta\cdot
y}}{\partial\nu(y)}-ie^{-i\omega\theta\cdot
y}\right][\psi_1(y)-\tilde{\psi}(y)]\
dy\\
& +\tau_0\int_{\Gamma_\rho}\left[\frac{\partial
e^{-i\omega\theta\cdot y}}{\partial\nu(y)}-i\kappa
e^{-i\omega\theta\cdot y}\right]\psi_2(y)\ dy\\
& +\omega^2\tau_0\int_{U_\rho}e^{-i\omega\theta\cdot
y}[1-q(y)]\tilde{v}(y)\ dy.
\end{split}
\end{equation}
Applying Lemma~\ref{lem:key lemma} and Remark~\ref{rem:2d case} to
estimating the terms in (\ref{eq:farfield difference}), along with
straightforward calculations, we have (\ref{eq:farfield
difference0}). The proof is completed.
\end{proof}

We are in a position to prove Theorem~\ref{thm:approximate
cloaking}.

\begin{proof}[Proof of Theorem~\ref{thm:approximate cloaking}]
Let $u\in H_{loc}^1(\mathbb{R}^n\backslash(\bar{P}\cup\bar{G}))$ be the
scattering wave field corresponding to
$\mathcal{E}=\bigoplus_{k=1}^l\{D_k\backslash\bar{P}_k;\sigma_{k,c},\eta_{k,c}\}\oplus
P\oplus G\oplus\{\mathbf{B}\backslash(\bar{D}\cup\bar{G});\sigma_0,\eta_0\}$, and
$\tilde{u}\in H_{loc}^1(\mathbb{R}^n\backslash\bar{G})$ be the scattering wave field
corresponding to $G\oplus\{\mathbf{B}\backslash\bar{G}; \sigma_0,\eta_0\}$. Define
$F:\mathbb{R}^n\backslash(\bar{U}_\rho\cup\bar{G})\rightarrow\mathbb{R}^n\backslash(\bar{P}\cup\bar{G})$
by
\begin{equation}
F:=\begin{cases}
& F_{k,\rho}\quad\mbox{on\ $\bar{D}_k\backslash U_\rho$},\\
& \mbox{Identity}\quad \mbox{otherwise}.
\end{cases}
\end{equation}
Obviously, $F$ is bi-Lipschitz and orientation-preserving. By the
transformation invariance of the Helmhlotz equation as we discussed
in Section~\ref{sect:3}, one sees that $F^*u\in
H_{loc}^1(\mathbb{R}^n\backslash(\bar{U}_\rho\cup\bar{G}))$ is the scattering
wave field corresponding to
$\mathcal{T}=U_\rho\oplus G\oplus\{\mathbf{B}\backslash(\bar{U}_\rho\cup\bar{G});\sigma_0,\eta_0\}$.
Since $F$ is the identity outside $D=\cup_{k=1}^l D_k$, we know
\begin{equation}\label{eq:F2}
A(\theta,\xi;\mathcal{E})=A(\theta,\xi;\mathcal{T}).
\end{equation}
According to our discussion in Lemma~\ref{lem:well posedness 1} and
Remark~\ref{rem:regular system}, we know $\tilde{u}\in
C^2(\mathbb{R}^n\backslash\bar{G})$ and $F^*u\in
C^2(\mathbb{R}^n\backslash(\bar{U}_\rho\cup\bar{G}))\cap C(\mathbb{R}^n\backslash
(U_\rho\cup G))$. Then, by Theorem~\ref{thm:small}
\[
\begin{split}
&\|A(\theta,\xi; \mathcal{T})-A(\theta,\xi; G\oplus\{\mathbf{B}\backslash\bar{G};\sigma_0,\eta_0\})\|_{L^2(\mathbb{S}^{n-1})}\\
=&\mathcal{O}(e(\rho)) \ \mbox{as\ $\rho\rightarrow 0^+$},
\end{split}
\]
which together with (\ref{eq:F2}) yields
\[
\|A(\theta,\xi;\mathcal{E})-A(\theta,\xi; G\oplus\{\mathbf{B};\sigma_0,\eta_0\})\|_{L^2(\mathbb{S}^{n-1})}=\mathcal{O}(e(\rho)).
\]

The proof is completed.
\end{proof}


%


\begin{thebibliography}{00}
%
\bibitem{AE} {A. Alu and N. Engheta}, {\it Achieving transparency with plasmonic and
metamaterial coatings}, Phys. Rev. E, {\bf 72} (2005), 016623.

\bibitem{HamKan} {H. Ammari and H. Kang}, {\it Reconstruction of Small Inhomogeneities
from Boundary Measurements}, Lecture Notes in Math. 1846,
Springer-Verlag, Berlin, 2004.

\bibitem{Amm} {H. Ammari and J.-C. N\'ed\'elec}, {\it Full low-frequency asymptotics for the reduced wave equation}, Appl. Math. Lett., \textbf{12} (1999), 127--131.

\bibitem{ColKre} {D. Colton and R. Kress}, {\it Inverse
Acoustic and Electromagnetic Scattering Theory}, 2nd Edition,
Springer-Verlag, Berlin, 1998.

\bibitem{ColKre2} {D. Colton and R. Kress}, {\it Integral Equation Method in Scattering Theory},
Wiley, New York, 1983.

\bibitem{FreVog} {A. Friedman and M. Vogelius}, {\it Identification of small inhomogeneities of extreme conductivity by boundary measurements: a theorem on continuous dependence}, Arch. Ration. Mech. Anal., \textbf{105} (1989), 299--326.

\bibitem{GKLU_2}
\newblock {A. Greenleaf, Y. Kurylev, M. Lassas and
G. Uhlmann},
\newblock \emph{Isotropic transformation optics: approximate acoustic and quantum cloaking},
\newblock New J. Phys., \textbf{10} (2008), 115024.




%





\bibitem{GKLU4} {A. Greenleaf, Y. Kurylev, M. Lassas and G. Uhlmann}, {\it Invisibility and inverse prolems}, Bulletin A. M. S.,
 \textbf{46} (2009), 55--97.

\bibitem{GKLU5} {A. Greenleaf, Y. Kurylev, M. Lassas and G. Uhlmann,
} {\it Cloaking devices, electromagnetic wormholes and transformation
optics}, SIAM Review, {\bf 51} (2009), 3--33.




\bibitem{GLU2} {A. Greenleaf, M. Lassas and G. Uhlmann},
{\it On nonuniqueness for Calder\'on's inverse problem}, Math. Res. Lett.,
{\bf 10} (2003), 685.








\bibitem{Isa2} {V. Isakov}, {\it On uniqueness in the inverse
transmission scattering problem}, Commun. Part. Diff. Eqn.,
\textbf{15} (1990), 1565--87.


\bibitem{Isa} {V. Isakov}, {\it Inverse Problems for Partial
Differential Equations}, 2nd ed., Springer-Verlag, New York, 2006.

\bibitem{KOVW} {R. Kohn, D. Onofrei, M. Vogelius and M. Weinstein}, {\it
Cloaking via change of variables for the Helmholtz equation},
Commu. Pure Appl. Math., {\bf 63} (2010), 0973--1016.


\bibitem{KSVW} {R. Kohn, H. Shen, M. Vogelius and M. Weinstein}, {\it Cloaking via change of variables in electrical impedance
tomography}, Inverse Problems, {\bf 24} (2008), 015016.


\bibitem{LaxPhi} {P. Lax and R. Phillips}, {\it Scattering
Theory}, Academic Press, Inc., San Diego, 1989.

\bibitem{Leo} {U. Leonhardt}, {\it Optical conformal mapping},
Science, {\bf 312} (2006), 1777--1780.


\bibitem{Liu} {H. Y. Liu},
{\it Virtual reshaping and invisibility in obstacle scattering}, Inverse
Problems, {\bf 25} (2009), 045006.

\bibitem{KirPai} {A. Kirsch and L. P\"av\"arinta}, {\it On
recovering obstacles inside inhomogeneities}, Math. Meth. Appl. Sci.,
\textbf{21} (1998), 619--651.

\bibitem{McL} {W. McLean},  {\it Strongly Elliptic Systems and Boundary Integral
Equations}, Cambridge University Press, Cambridge, 2000.
%
%
%



\bibitem{MN} {G. W. Milton and N.-A. P. Nicorovici},
{\it On the cloaking effects associated with anomalous localized resonance},
Proc. Roy. Soc. A, {\bf 462} (2006), 3027--3095.







%


\bibitem{Nor}
{A. N. Norris}, {\it Acoustic cloaking theory}, Proc. R. Soc. A,
{\bf 464} (2008), 2411--2434.

\bibitem{PenSchSmi} {J. B. Pendry, D. Schurig and D. R. Smith}, {\it Controlling electromagnetic fields}, Science, {\bf 312} (2006),
1780--1782.

\bibitem{RYNQ} {Z. Ruan, M. Yan, C. W. Neff and M. Qiu},
{\it Ideal cylyindrical cloak: Perfect but sensitive to tiny
perturbations}, Phy. Rev. Lett., \textbf{99} (2007), no. 11, 113903.




\bibitem{U} {G. Uhlmann}, {\it Calder\'on's problem and electric impedance tomography}, Inverse Problems, \textbf{25} (2009), 123011.


\bibitem{U2} {G. Uhlmann}, {\it Visibility and invisibility},
ICIAM 07--6th International Congress on Industrial and Applied Mathematics,
381–408, Eur. Math. Soc., Z\"urich, 2009,









\bibitem{YYQ}
{M. Yan, W. Yan and M. Qiu}, {\it Invisibility cloaking by
coordinate transformation}, Chapter 4 of {\it Progress in
Optics}--Vol. 52, Elsevier, pp. 261--304, 2008.

%


\end{thebibliography}
\end{document}